%% file: smt_for_few_targets.tex
\documentclass[english, 11pt]{amsart}

\usepackage{tikz}
\usepackage{aliascnt}
\usepackage{mathrsfs}
\usepackage[all,poly,knot]{xy}
\usepackage{hyperref}
\usepackage{comment}  
\usepackage{csquotes}  
\usepackage{amssymb,amsbsy,amsmath,amsfonts,amssymb,amscd,
	graphics,color,footmisc,fancyhdr,multicol,fancybox,
	graphicx,mathrsfs,rotating,ifthen,wasysym,dsfont}

\usepackage{times}
\usepackage{pdfpages}
\usepackage{multirow}
\usepackage{nccrules}
\usepackage{textcomp}
\usepackage{comment}
\usepackage{framed}
\usepackage[all]{xy}
\usepackage{graphicx}
\usepackage{pgf,tikz}
\usepackage{mathrsfs}
\usetikzlibrary{arrows}
\usepackage{lipsum}
\usepackage{mathtools}

\newcommand{\medcap}{\mathbin{\scalebox{1.5}{\ensuremath{\cap}}}}

\DeclareMathOperator{\St}{St}

\DeclareMathOperator{\Sing}{Sing}

\DeclareMathOperator{\dif}{d}

\DeclareMathOperator{\rank}{rank}
\DeclareMathOperator{\lcm}{lcm}

\DeclareMathOperator{\codim}{codim}

\DeclareMathOperator{\ord}{ord}

\DeclareMathOperator{\ns}{ns}

\def\log{\mathrm{log}\,}

\theoremstyle{plain}

\newtheorem{thm}{Theorem}[section]  

\newtheorem{cor}[thm]{{Corollary}} 

\newtheorem{lem}[thm]{{Lemma}}

\newtheorem{pro}[thm]{Proposition}

\theoremstyle{remark}
\newtheorem{rmk}[thm]{Remark}

\numberwithin{equation}{section}

\usepackage[top=1.5in, bottom=1.5in, left=1in, right=1in]{geometry}

 \usepackage[hyperpageref]{backref}
 
\usepackage{xcolor}
\hypersetup{
	colorlinks,
	linkcolor={red!50!black},
	citecolor={blue!62!black},
	urlcolor={blue!80!black}
}
\theoremstyle{plain}
\newcommand{\thistheoremname}{}
\newtheorem*{genericthm*}{\thistheoremname}
\newenvironment{namedthm*}[1]{\renewcommand{\thistheoremname}{#1}%
	\begin{genericthm*}}
	{\end{genericthm*}}

\newtheoremstyle{named}{}{}{\itshape}{}{\bfseries}{.}{.5em}{\thmnote{#3's }#1}
\theoremstyle{named}

\makeatletter
\newcommand\thankssymb[1]{\textsuperscript{\@fnsymbol{#1}}}
\makeatother
\begin{document} 
\title[Entire holomorphic curves into $\mathbb{P}^n(\mathbb{C})$]{\bf Entire holomorphic curves into $\mathbb{P}^n(\mathbb{C})$
	\\
	intersecting $n+1$ general hypersurfaces
}


\author{Zhangchi Chen}
\address{Morningside Center of Mathematics, Academy of Mathematics and System Science, Chinese Academy of Sciences, Beijing 100190, China}
\email{zhangchi.chen@amss.ac.cn}

\author{Dinh Tuan Huynh}

\address{Department of Mathematics, University of Education, Hue University, 34 Le Loi St., Hue City, Vietnam}
\email{dinhtuanhuynh@hueuni.edu.vn}

\author{Ruiran Sun}
\address{Department of Mathematics \& Statistics, McGill University, Burnside Hall
805 Sherbrooke Street West
Montreal, Quebec H3A 0B9}
\email{ruiran.sun@mcgill.ca}

\author{Song-Yan Xie}
\address{Academy of Mathematics and System Science \& Hua Loo-Keng Key Laboratory
	of Mathematics, Chinese Academy of Sciences, Beijing 100190, China;  
		School of Mathematical Sciences, University of Chinese Academy of Sciences, Beijing
100049, China}
\email{xiesongyan@amss.ac.cn}

\begin{abstract}
	Let $\{D_i\}_{i=1}^{n+1}$ be $n+1$  hypersurfaces in $\mathbb{P}^n(\mathbb{C})$ with total degrees $\sum_{i=1}^{n+1} \deg D_i\geqslant n+2$, in general position and satisfying a generic geometric condition: every $n$ hypersurfaces intersect only at smooth points and the intersection is transversal. Then, for every algebraically non-degenerate entire holomorphic curve $f\colon\mathbb{C}\rightarrow\mathbb{P}^n(\mathbb{C})$, 
	we show a Second Main Theorem:
	\[
 \sum_{i=1}^{n+1} \delta_f(D_i)
<
 n+1
	\]
	in terms of defect inequality in Nevanlinna theory. This is the first result in the literature on Second Main Theorem for $n+1$ general hypersurfaces in $\mathbb{P}^n(\mathbb{C})$ with optimal total degrees. 
\end{abstract}

\maketitle
%\tableofcontents 

\section{\bf Introduction}
Given a codimension one subvariety $D$ in a complex manifold $X$ such that
the complement $X\setminus D$ satisfies certain complex hyperbolicity quality in spirit of the Kobayashi conjecture~\cite{Kobayashi1970} or the Green-Griffiths conjecture~\cite{Green-Griffiths1980}, one seeks to reach a quantitative strengthening in terms of Second Main Theorem in Nevanlinna theory, which bounds in certain proportional way,  the ``growth rate'' of an algebraically nondegenerate  holomorphic map $f\colon S\rightarrow X$ from certain source space $S$ usually  being $\mathbb{C}$,  from above by the ``intersection frequency''  or ``impact'' of $f(S)$ with respect to $D$.

The classical example is Nevanlinna's celebrated work~\cite{Nevanlinna1925} which quantifies the

\begin{namedthm*}{Little Picard Theorem} If $p_1,p_2,p_3$ are three distinct points in $\mathbb{P}^1(\mathbb{C})$, then any meromorphic function $f\colon\mathbb{C}\rightarrow\mathbb{P}^1(\mathbb{C})\backslash\{p_1,p_2,p_3\}$ is constant.
\qed
\end{namedthm*}

For higher dimensional target space $X$,
for various source spaces $S$ and  holomorphic maps $f\colon S\rightarrow X$, we refer the readers to~\cite{Nevanlinna-book, Stoll77, Fujimoto-book, Noguchi-Winkelmann2014, Ru2021book} for later developments.
The leading problem is this direction is the following

%Higher dimensional generalizations of the Little Picard Theorem expect the degeneracy of the degeneracy of holomorphic curves from $\mathbb{C}$ into a projective space $\mathbb{P}^n(\mathbb{C})$ omitting a hypersurface $D$ of large enough degree. Green and Griffiths \cite{Green-Griffiths1980} anticipated that in this case, the expected lower degree bound of $D$ should be $n+2$.

\begin{namedthm*}{Fundamental Conjecture of Entire Curves}[cf.~\cite{Griffiths72, Noguchi-Winkelmann2014}]   Let $D$ be a simple normal crossing divisor on the projective space $\mathbb{P}^n(\mathbb{C})$ of degree $d\geqslant n+2$. Let $f\colon\mathbb{C}\rightarrow\mathbb{P}^n(\mathbb{C})$ be an entire holomorphic curve. If the image of $f$ is not contained in any hypersurface, then the following Second Main Theorem type estimate holds
\begin{equation}
\label{fundamental conj for entire curves}
(d-n-1)\,T_f(r)\leqslant N^{[k_0]}_f(r,D)+o\big(T_f(r)\big)
\qquad\parallel,
\end{equation}
where $k_0\in\mathbb{N}$ is a positive integer independent of $f$. 
\end{namedthm*}

Here $N^{[k_0]}_f(r,D)$ and $T_f(r)$ are standard notions in Nevanlinna theory, which will be introduced in the next paragraph. For non-negatively valued functions $\phi(r)$, $\psi(r)$ defined for $r\geqslant r_0\geqslant0$, we write
\[
\phi(r)\leqslant \psi(r)
\qquad\parallel
\]
if the inequality holds for $r\geqslant r_0$ outside a Borel set of finite Lebesgue measure.

\medskip
 Let $f\colon\mathbb{C}\rightarrow\mathbb{P}^n(\mathbb{C})$ be an entire holomorphic curve and let $D\subset\mathbb{P}^n(\mathbb{C})$ be a hypersurface such that $f(\mathbb{C})\not\subset D$. The order function
\[
T_f(r)
\,
:=
\,
\int_1^r\dfrac{\dif  t}{t} \int_{\mathbb{D}_t}f^*\omega_{FS}
    \qquad
{{\scriptstyle (r\,>\,1)},}
\]
is a geometric equivalent version of Nevanlinna's characteristic function, historically discovered independently by Shimizu and Ahlfors~\cite[pp.~11--12]{Noguchi-Winkelmann2014},
measuring the area growth of the image of the disc $\mathbb{D}_r$ centered at $0$ with radius $r$, with respect to the Fubini--Study metric $\omega_{FS}$. 
For $k\in\mathbb{N}\cup\{\infty\}$,
the level--$k$ truncated counting function 
\[
N_f^{[k]}(r,D)
\,
:=
\,
\int_1^r
\frac{\dif t}{t} \sum_{|z|<t}\min\{k,\ord_zf^*D\}
\]
captures the intersection frequencies of         $f(\mathbb{C})\cap D$. The defect of $f$ with respect to $D$ is given by
\[
\delta_f^{[k]}(D):=
\liminf_{r\rightarrow\infty}
\bigg(
1-\dfrac{N_f^{[k]}(r,D)}{\deg(D)\,T_f(r)}\bigg)
.
\]
For brevity, when $k=\infty$, we write $N_f(r,D)$, $\delta_f(D)$  instead of $N_f^{[\infty]}(r,D)$, $\delta_f^{[\infty]}(D)$ .

The {\sl First Main Theorem} in Nevanlinna theory, which is a reformulation of the Lelong-Jensen formula, provides an upper bound for the counting function in terms of the order function:
\[
N_f(r,D)
\leqslant
d\,T_f(r)
+
O(1),
\]
which implies
\begin{equation}
\label{defect-FMT}
0\leqslant \delta_f(D)\leqslant 1.
\end{equation}

The reverse direction, i.e.,  bounding the order function from above by the sum of counting functions of many divisors, is usually much harder. Such type of results are called {\sl Second Main Theorems}. The question of establishing a satisfactory estimate of the form \eqref{fundamental conj for entire curves} is still very open in general.  When $D$ consists of $q\geqslant n+2$ hyperplanes $H_i\subset \mathbb{P}^n(\mathbb{C})$ in general position ($1\leqslant i\leqslant q$), and $f: \mathbb{C}\rightarrow \mathbb{P}^n(\mathbb{C})$ is linearly nondegenerate, such a Second Main Theorem with truncation at level $n$ is established by H.~Cartan \cite{Cartan1933}, which yields the following {\sl defect relation}
\[
\sum_{i=1}^q\delta_f^{[n]}(H_i)\leqslant n+1.
\]
When all components  of $D$ are hypersurfaces and the image of $f$ is not contained in $D$, a Second Main Theorem  without effective truncation level \cite{Eremenko-Sodin1992}  were obtained by Eremenko-Sodin via potential theoretic method, which implies a defect relation bounded by $2n$. Assuming furthermore that $f$ is algebraically nondegenerate, a stronger estimate~\cite{Minru2004} was established by Ru, which yields a defect relation bounded by $n+1$. 
In \cite{Griffiths72}, Griffiths conjectured~\eqref{fundamental conj for entire curves} for $k_0=\infty$ in the right-hand side. 
This conjecture quantifies

\medskip
\begin{namedthm*}{
Logarithmic Green-Griffiths' Conjecture}[\cite{Green-Griffiths1980}] If $D$ is a simple normal crossing divisor on the projective space $\mathbb{P}^n(\mathbb{C})$ of degree $d\geqslant n+2$, then the image of any holomorphic curve $f:\mathbb{C}\rightarrow\mathbb{P}^n(\mathbb{C})$ omitting $D$ lies in some proper algebraic subvariety of $\mathbb{P}^n(\mathbb{C})$.
\end{namedthm*}

When $D$ has $q\leqslant n+1$ components,  few  Second Main Theorem type results toward~\eqref{fundamental conj for entire curves}  were known. In hindsight,
the difficulty is intimately related to establishing the (conjectured) hyperbolicity property of $X\setminus D$  
using the jet differential technique introduced by Bloch~\cite{Bloch26}.
Indeed, on one hand, the logarithmic fundamental vanishing theorem of entire curves states that,
any negatively twisted logarithmic (along $D$) $k$-jet differential  $\omega$ serves as an obstruction for the existence of entire curves
$f: \mathbb{C}\rightarrow X\setminus D$, since $f$ must obey the differential equation $f^*\omega\equiv 0$~(cf. e.g.~\cite{Ru2021book}). 
On the other hand, it is shown in
~\cite[Theorem~3.1]{HVX17}
that, for any
negatively twisted logarithmic (along $D$) $k$-jet differential $\omega$ and for any entire curve $f: \mathbb{C}\rightarrow X$ not contained in $D$, if $f^*\omega\not\equiv 0$, then one can obtain a Second Main Theorem (SMT) for the entire curve $f$ with respect to the divisor $D$.
 Thus, for showing the hyperbolicity of $X\setminus D$,
 or for obtaining a SMT of $f$ with respect to $D$, one tries
 to find sufficiently many negatively twisted logarithmic (along $D$) $k$-jet differentials $\{\omega_i\}_{i=1}^M$ having  ``tiny'' common base loci supporting no entire curve therein.

However, in practice, such approach is very difficult. 
 For instance, in the simplest case that $k=1$ and $D=\varnothing$, there was a related conjecture of Debarre~\cite{Debarre-conjecture} anticipating that, 
for general  $c\geqslant n/2$ hypersurfaces $H_1, \dots, H_c \subset \mathbb{P}^n(\mathbb{C})$ with large
degrees $\gg 1$, 
the intersection $X:= H_1\cap \cdots\cap H_c$ shall have
ample cotangent bundle $T_X^*$. The Debarre ampleness conjecture was first proved  in~\cite{Xie-Invent} (arXiv:1510.06323),
in which the difficulty of controlling the base loci was settled by ad hoc  symmetry of certain sophisticated deformed Fermat type polynomial equations,
using explicit $1$-jet differentials obtained in~\cite{Brotbek-MathAnn}. See also another proof~\cite{Brotbek-Darondeau} (arXiv:1511.04709) appeared shortly later using more symmetric generalized Fermat type polynomial equations.

In the vein of Siu's strategy~\cite{Siu04} for the  Kobayashi and Green-Griffiths conjectures, namely by  using  slanted vector fields~\cite{Siu02, Merker2009, Darondeau_vectorfield} and certain Riemann-Roch calculation~\cite{Darondeau2016}, a Second Main Theorem~\cite{HVX17} was established in the case $q=1$  for algebraically nondegenerate entire curve $f: \mathbb{C}\rightarrow \mathbb{P}^n(\mathbb{C})$ with respect to a general hypersurface $D\subset \mathbb{P}^n(\mathbb{C})$ of  large degree $d\geqslant 15\,(5\,n+1)\,n^n$. Thanks to the breakthrough~\cite{Riedl-Yang22} of Riedl and Yang, one can remove the Zariski dense assumption on $f(\mathbb{C})\subset \mathbb{P}^n(\mathbb{C})$. Moreover, the exponential degree bound can be improved to some polynomial bound $O(n^4)$ by the recent advancement of B\'erczi and Kirwan~\cite{Berczi-Kirwan-2023-1,Berczi-Kirwan-2023-2}. 

From now on,  finitely many hypersurfaces $D_1,\dots,D_q$, $q\geqslant n$, are said to be {\em intersecting transversally}, if for any $n$ hypersurfaces $D_{j_1},\dots,D_{j_n}$, and for any $z$ in their intersection, we have
\begin{itemize}
\item $z$ is {\em a smooth point} of each $D_{j_k}$, $1\leqslant k\leqslant n$;
\item {\em the normal vectors} of the tangent spaces $T_z D_{j_k}$, $1\leqslant k\leqslant n$, are {\em linearly independent}.
\end{itemize}

In this paper, we study the case that 
$
D
=
\cup_{i=1}^{q}
D_i$ 
consists of $q= n+1$ hypersurfaces $D_i\subset\mathbb{P}^n(\mathbb{C})$ (not all being hyperplanes) in general position and intersecting transversally.
The algebraic degeneracy of entire holomorphic curves into the complement
$\mathbb{P}^n(\mathbb{C})\setminus D$ 
was established by Noguchi-Winkelmann-Yamanoi \cite{NWA2007}. See~\cite[Theorem~1.6]{Guo-Sun-Wang-2021} and \cite[Theorem~1.2]{Guo-Sun-Wang-2022} for  moving target versions. Quantitatively, we obtained a Second Main Theorem.

\begin{namedthm*}{Main Theorem}
Let $\{D_i\}_{i=1}^{n+1}$ be $n+1$  hypersurfaces in $\mathbb{P}^n(\mathbb{C})$ with total degrees $\sum_{i=1}^{n+1} \deg D_i\geqslant n+2$, in general position and intersecting transversally. Then, for every algebraically nondegenerate entire holomorphic curve $f\colon\mathbb{C}\rightarrow\mathbb{P}^n(\mathbb{C})$, the following defect inequality holds 
	\begin{equation}
	\label{defect relation}
	\sum_{i=1}^{n+1} \delta_f(D_i)
	< n+1.
	\end{equation}
\end{namedthm*}

Clearly, $\delta_f(D_i)=1$ if and only if \begin{equation}
\label{defect=1}
N_f(r,D_i)=o\big(T_f(r)\big)
\qquad
{{\scriptstyle (r\,\rightarrow\,\infty)}},
\end{equation}
literally, the curve $f$ does not meet $D_i$ often. Theorefore~\eqref{defect relation} serves as a weak Second Main Theorem. 

As a matter of fact, our initial motivation is to study the case of $3$ conics in $\mathbb{P}^2(\mathbb{C})$~\cite{Grauert-Peternell1985, dethloff_schumacher_wong1995}.
%See also~\cite{babets1984, Zaidenberg1988, 	siu_yeung1996,BD2001, Rousseau2009, Tiba2013} for nearby hyperbolicity results.
Even in this simple case, the aforementioned methods  seem infertile.  

Back to our main theorem, we will take an alternative geometric approach in which the number 
\[
n+1=\dim_{\mathbb{C}} \mathbb{P}^n(\mathbb{C})+1
\] 
of  components of $D$ is critical. 
Let us sketch the proof now.
For simplicity, we assume that every hypersurface 
$
D_i\subset \mathbb{P}^n(\mathbb{C})
$ is defined by some homogeneous polynomial $Q_i\in \mathbb{C}[z_0, \dots, z_n]$ of equal degree $d$.
Suppose on the contrary that~\eqref{defect relation} fails, i.e., by~\eqref{defect-FMT}, 
all defect values reach maximum
\begin{equation}
\label{presumed condition} \delta_f(D_i)=1 
\qquad
{\scriptstyle (i\,=\,1, \,\dots,\, n+1)}.  
\end{equation}
For the parabolic Riemann surface
$\mathbb{C}\setminus f^{-1}(D)$, 
we will employ an exhaustion function $\sigma$ such that 
the weighted Euler characteristic $\mathfrak{X}_{\sigma}(r)$ is negligable 
\begin{equation}
\limsup_{r\rightarrow \infty}\dfrac{\mathfrak{X}_{\sigma}(r)}{T_{f, \sigma}(r)}
=
0
\end{equation}
compared with the parabolic order function $T_{f,\sigma}(r)$ (see Section~\ref{section semiabelian}).

The key trick is introducing the auxiliary
hypersurface $\mathcal{V}\subset \mathbb{P}^n(\mathbb{C})$  defined by the Jacobian
\[
\det
\dfrac{\partial (Q_1, \dots, Q_{n+1})}{\partial (z_0, \dots, z_{n})}
\]
of degree 
$
\sum_{i=0}^n d_i-(n+1)$.
Such hypersurface $\mathcal{V}$ was used in \cite[p.~261]{Francois-Duval-2001} and~\cite{MR2369088} for $n=2$, and in~\cite{Guo-Sun-Wang-2021, Guo-Sun-Wang-2022} for general $n$. 
Geometrically, $\mathcal{V}$
consists of the critical points of the endomorphism 
\begin{equation*}
F(z)=\big[Q_1(z):Q_2(z):\dots:Q_{n+1}(z)\big]\quad
\colon\quad \mathbb{P}^n(\mathbb{C})\longrightarrow\mathbb{P}^n(\mathbb{C}).
\end{equation*}
Whence if the entire curve $f$ intersects $\mathcal{V}$ at a point $P\in \mathbb{C}$, the composition $g:=F\circ f$ must be tangent to $\mathcal{W}:=F(\mathcal{V})$, i.e., having intersection multiplicity~$\geqslant 2$ at  $P$.
For $\{H_i\}_{i=1}^{n+1}$ in general position, the hypersurface $\mathcal{V}$ is {\em in general position with} $\{H_i\}_{i=1}^{n+1}$, {\em i.e.}, $\mathcal{V}$ and any $n$ hypersurfaces among $\{H_i\}_{i=1}^{n+1}$ have empty intersection, {\em if and only if} $\{H_i\}_{i=1}^{n+1}$ are intersecting transversally. The {\em if} part is provided in~\cite[Section~5]{Guo-Sun-Wang-2022}. We will prove the {\em only if} part in Lemma~\ref{lem:transversal}. Thus we can apply a Second Main Theorem of  Ru~\cite{Minru2004} to show that,
under the presumed condition~\eqref{presumed condition},
the intersection frequency of the holomorphic curve $\widetilde{f}:=f|_{\mathbb{C}\setminus f^{-1}(D)}$ with
$\mathcal{V}$ must be high. This will contradict  another fact, to be obtained in Section~\ref{section semiabelian} following a strategy of
Noguchi-Winkelmann-Yamanoi~\cite{Noguchi-Yamanoi-Winkelmann2008},
that the parabolic holomorphic curve $\widetilde{g}:=g|_{\mathbb{C}\setminus f^{-1}(D)}$
into the semi-abelian variety $(\mathbb{C}^*)^n\subset\mathbb{P}^n(\mathbb{C})$ cannot be 
tangent to the effective divisor $\mathcal{W}$ very often. For details of proofs, see Sections~\ref{sect-smoothing} and~\ref{Proof of the main theorem}.

\section*{Acknowledgement}
We learned the question of seeking a Second Main Theorem in the presence of three conics in the projective plane from Julien Duval
and from him we got fruitful ideas, e.g. the hypersurface $\mathcal{V}$. Here we would like to address our profound gratitude to him. We are also grateful to Junjiro Noguchi for sharing his interesting ideas and private note with us. 
We thank Min Ru for his comments and valuable suggestions on an early manuscript, and for drawing our attention to~\cite{MR2369088, Guo-Sun-Wang-2021, Guo-Sun-Wang-2022}.  
We thank Xiangyu Zhou for his remarks during a seminar talk.
S.-Y. Xie is partially supported by National Key R\&D Program of China Grant No.~2021YFA1003100 and NSFC Grant No.~12288201.
D.T. Huynh  is supported by the Vietnam Ministry of Education and Training under grant number B2024-DHH-01, and
a part of this article was written while he was
visiting Vietnam Institute for Advanced Study in Mathematics (VIASM).
R. Sun thanks the CRM, Montreal, for its support during his CRM-postdoctoral fellowship and to McGill University for its hospitality.
Z. Chen is supported in part by the Labex CEMPI (ANR-11-LABX-0007-01), the project QuaSiDy (ANR-21-CE40-0016), and China Postdoctoral Science Foundation (2023M733690).

\section{\bf Parabolic Nevanlinna theory in semi-abelian varieties and projective spaces}
\label{section semiabelian}

A non-compact Riemann surface $\mathcal Y$ is called {\it parabolic} if it admits a smooth
exhaustion function 
\[
\sigma\colon {\mathcal Y}\to [1, \infty[
\]
such that $\log \sigma$ is harmonic outside a compact subset of 
$\mathcal{Y}$. For every $r>1$, we denote by 
\[
B^\sigma_r:= \big\{z\in \mathcal Y : \sigma(z)< r\big\},\qquad\qquad
S^\sigma_r:= \big\{z\in \mathcal Y : \sigma(z)= r\big\},
\]
the open {\sl parabolic ball} and the {\sl parabolic sphere} of radius $r$ respectively.
By Sard's theorem,
for almost every value $r\in \mathbb{R}_{>1}$, the sphere $S^\sigma_r$ is
smooth. We donote the  Euler characteristic
of $B^\sigma_r$ by $\chi_{\sigma} (r)$,
and we consider the induced length measure
\[
\dif\mu_r:= \dif^c\log \sigma|_{S^\sigma_r},
\]
where $\dif^c:=\frac{\sqrt{-1}}{4\pi}(\bar{\partial}-\partial)$.
 The {\sl weighted Euler characteristic} $\mathfrak{X}_{\sigma}(r)$ is then defined by logarithmic average
\[
\mathfrak{X}_{\sigma}(r):=\int_1^r\chi_{\sigma}(t)\,\dfrac{\dif t}{t}
\qquad
{\scriptstyle (r\,>\,1)}.
\]

Replacing the exhaustion $\mathbb{C}=\cup_{r>1}\,\mathbb{D}_r$ by $\mathcal{Y}=\cup_{r>1}\,B^{\sigma}_r$, one can develop Nevanlinna theory for parabolic Riemann surfaces (cf.~\cite{Stoll77, Sibony-Paun2021}). Let $X$ be a compact complex manifold. Let $L$ be a holomorphic line bundle on $X$ equipped with some Hermitian metric $\|\!\cdot\!\|$ with the Chern $(1, 1)$-form $\omega_L$. Let $E$ be an effective divisor defined by a global nonzero section $s$ of $L$. In the parabolic context, the standard notions in Nevanlinna theory are defined as follows.
\begin{itemize}
\item [1.] The $k$-truncated counting function
\[
N^{[k]}_{f, \sigma}(r, E)
\,
:=
\,
\int_1^r
\sum_{z\in B^\sigma_t}\min\{k,\ord_zf^*E\}
\,
\frac{\dif t}{t}
\qquad
{{\scriptstyle (k\,\in\, 
\mathbb{N}\,\cup \,\{\infty\};\, r\,>\,1)}}.
\]
\item [2.] The proximity function
\[
m_{f, \sigma}(r, E)
\,
:=
\,
\int_{S^\sigma_r}
\log
\dfrac{1}{\|s\circ f\|}
\,
\dif\mu_r
\qquad
{\scriptstyle (r\,>\,1)}.
\]
\item [3.] The order function
\[
T_{f, \sigma}(r,L)
\,
:=
\,
\int_1^r\dfrac{\dif t}{t}\int_{B^\sigma_t}f^*\omega_L
\qquad
{{\scriptstyle (r\,>\,1)}}.
\]
\end{itemize}
By Jensen's formula in the parabolic setting~\cite[Proposition~3.1]{Sibony-Paun2021}, one has the following 

\begin{namedthm*}{Parabolic First Main Theorem}
 Let $f\colon\mathcal{Y}\rightarrow X$ be a holomorphic map such that $f(\mathcal{Y})\not\subset\mathrm{Supp}(E)$. Then
\[
T_{f, \sigma}(r,L)
\,
=
\,
m_{f, \sigma}(r, E)
+
N_{f, \sigma}(r, E)
+O(1)
\qquad
{{\scriptstyle (r\,>\,1)}}.
\]	
\qed
\end{namedthm*}

For a parabolic Second Main Theorem, the weighted Euler characteristic naturally appears. Define the proximity function for the critical set as \cite[Definition~3.4]{Sibony-Paun2021}
\[
\mathfrak{X}^+_{ \sigma}(r)
:=
\int_{S^{\sigma}_r}\log^+|\text{d}\sigma(\tfrac{\partial}{\partial z})|^2\,\text{d}\mu_r.
\]
In \cite[pp.~32--33]{Sibony-Paun2021}, it is proved that for $\mathcal{Y}=\mathbb{C}\backslash\mathcal{E}$ with $\mathcal{E}=\{a_j\}_{j=1}^{\infty}$ a discrete countable set of  points in $\mathbb{C}$, one can take $r_j\in(0,1)$ sufficiently small such that
\begin{itemize}
\smallskip
\item[$\bullet$]
the discs $\mathbb{D}(a_j,2\,r_j)$ are disjoint,
\smallskip
\item[$\bullet$] the sum $\sum\limits_{j\geqslant 1}r_j<+\infty$.
\end{itemize}
For a smoothing $\sigma$ (see Section~\ref{sect-smoothing} or the Appendix for details) of the exhaustion function $\hat\sigma$ defined by
\[
\log\hat\sigma:=\log^+|z|+\sum\limits_{j\geqslant 1}r_j\,\log^+\tfrac{r_j}{|z-a_j|},
\]
we see that $\dif\dif^c\log\sigma$ is of finite mass and
\[
\mathfrak{X}^+_{ \sigma}(r)=\mathfrak{X}_{\sigma}(r)+O(\log r).
\]

\begin{namedthm*}{Parabolic Logarithmic Derivative Lemma} (\cite[Theorem~3.8]{Sibony-Paun2021})
Let $f\colon\mathcal{Y}\rightarrow\mathbb{P}^1(\mathbb{C})$ be a nonconstant meromorphic function. For any $\delta>0$, one has
\begin{align}\label{ldl}
m_{\frac{f'}{f}, \sigma}(r)
\leqslant
(1+\delta)^2\,\big(\log T_{f,\sigma}(r)\big)+(1+\delta)\,\log r+\mathfrak{X}^+_{ \sigma}(r)
+O(1)
\qquad\parallel.
\end{align}
\qed
\end{namedthm*}

In particular, when $\mathcal{Y}=\mathbb{C}\backslash\mathcal{E}$ with $\mathcal{E}=\{a_j\}_{j=1}^{\infty}$ a discrete countable set of points in $\mathbb{C}$, there exists some positive constant $C>0$ such that the following estimate
\[
m_{\frac{f'}{f}, \sigma}(r)
\leqslant
C\,\big(\log T_{f,\sigma}(r)+\log r\big)+\mathfrak{X}_{ \sigma}(r)
\qquad\parallel.
\]

Consequently, some results in the value distribution theory of entire holomorphic curves can be translated to the parabolic setting.

Throughout this section, we fix a smooth exhaustion $\sigma$ on the parabolic Riemann surface $\mathcal{Y}$. 
In \cite{Noguchi-Yamanoi-Winkelmann2008}, Noguchi-Winkelmann-Yamanoi established a Second Main Theorem type estimate for $k$-jet liftings of algebraically nondegenerate entire holomorphic curves $f$ into semi-abelian varieties with the optimal truncation level-one counting function, accepting an error term of the form $\epsilon\,T_{f,\sigma}(r)$, or equivalently $o\big(T_{f,\sigma}(r)\big)$, see \cite[Lemma~1.5]{Yamanoi2013}. This provides several applications in studying the degeneracy of holomorphic curves \cite{NWA2007,Noguchi-Yamanoi-Winkelmann2008}.

This remarkable result can be translated into the parabolic context, but we need to take into account the weighted Euler characteristic $\mathfrak{X}_\sigma(r)$ appearing each time when we apply the logarithmic derivative lemma. Hence from now on, we assume that
\begin{equation}
\label{assumption of weight euler charactersistic}
\limsup_{r\rightarrow\infty}\dfrac{\mathfrak{X}_\sigma(r)}{T_{f,\sigma}(r)}
=
0.
\end{equation}

For our Main Theorem, we only need to deal with parabolic holomorphic curves in $(\mathbb{C}^*)^n$. Nevertheless, we must use higher order jets and establish a Second Main Theorem type estimate, not only for divisors, but also for subvarieties of codimension $\geqslant 2$ (cf.~\cite[Section~2.4.1]{Noguchi-Winkelmann2014}). For the notions and the properties of logarithmic $k$-jet bundles, we refer the readers to \cite{Noguchi1986, Dethloff-Lu2001}.

Under the assumption~\eqref{assumption of weight euler charactersistic}, we can translate the result of~\cite[Theorem~6.5.1]{Noguchi-Winkelmann2014}
for the special case $A:=(\mathbb{C}^*)^n$ in the parabolic context as follows. 
\begin{thm}
	\label{smt-semiabelian, truncation level1}
	Let $\mathcal{Y}$ be a parabolic Riemann surface with an exhaustion function $\sigma$. Let $f\colon\mathcal{Y}\rightarrow A:=(\mathbb{C}^*)^n$ be an algebraically nondegenerate holomorphic curve. For an integer $k\geqslant 0$, denote by $J_kf$ the $k$-jet lifting of $f$ and by $X_k(f)$ the Zariski closure of $J_kf$ in the $k$-jet space $J_k(A)$.  Let $Z$ be an algebraic reduced subvariety of $X_k(f)$. 
	\begin{enumerate}
		\item 
		There exists a compactification $\bar{X}_k(f)$ of $X_k(f)$  such that
		\begin{equation*}
		T_{J_kf,\sigma}(r,\omega_{\bar{Z}})\leqslant
		N^{[1]}_{J_kf,\sigma}(r,Z)
		+
		o\big(T_{f,\sigma}(r)\big)
		\,\,\parallel,
		\end{equation*}
		where $\bar{Z}$ denotes the closure of $Z$ in $\bar{X}_k(f)$.
		\item	
		Assume furthermore that $\codim_{X_k(f)}Z\geqslant 2$, then 
		\begin{equation*}
		T_{J_kf,\sigma}(r,\omega_{\bar{Z}})
		=
		o\big(T_{f,\sigma}(r)\big)
		\,\,\parallel.
		\end{equation*}
		\item In the case where $k=0$ and $Z$ is an effective divisor $D$ on $A$, there exists a smooth compactification of $A$ independent of $f$,  such that
		\[
		T_{f,\sigma}\big(r,L(\overline{D})\big)
		\leqslant
		N^{[1]}_{f,\sigma}(r,D)
		+
		o\Big(T_{f,\sigma}\big(r,L(\overline{D})\big)\Big)
		\,\,\parallel.
		\] 
	\end{enumerate}
\end{thm} 

This together with the First Main Theorem yields the following

\begin{cor}
	\label{tangency not often}
	Let $\mathcal{Y}$ be a parabolic Riemann surface with an exhaustion function $\sigma$. Let $D$ be an effective divisor on  $A:=(\mathbb{C}^*)^n$. Let $f\colon\mathcal{Y}\rightarrow A$ be an algebraically nondegenerate holomorphic map. Then there exists a smooth compactification of $A$ independent
	of $f$, such that
	\[
	N_{f,\sigma}(r,D)
	-
	N^{[1]}_{f,\sigma}(r,D)
	= o\Big(T_{f,\sigma}\big(r,L(\overline{D})\big)\Big)
	\,\,\parallel.
	\]
	\qed
\end{cor}

The proof of Theorem~\ref{smt-semiabelian, truncation level1} will be reached later in this section by implementing some modifications along the strategy of \cite{Noguchi-Yamanoi-Winkelmann2008}. First, we translate \cite[Lemma~4.7.1]{Noguchi-Winkelmann2014} to the parabolic context directly. 

\begin{lem}
	\label{parabolic lemma on logarithmic form}
	Let $M$ be a complex projective manifold and let $D$ be a reduced divisor on $M$. Let $f\colon\mathcal{Y}\rightarrow M$ be a holomorphic curve from a parabolic Riemann surface $\mathcal{Y}$ with an exhaustion function $\sigma$ into $M$ such that $f(\mathcal{Y})\not\subset D$. Let $\omega$ be a logarithmic (along $D$) $k$-jet differential  on $M$. Put $\xi:=\omega(J_kf)$. Then
	\[
	m_{\xi,\sigma}(r)\leqslant \mathfrak{S}_{f,\sigma}(r)+C\,\mathfrak{X}_\sigma(r)=o\big(T_{f,\sigma}(r)\big)
	\,\,\parallel,
	\]
where $\mathfrak{S}_{f,\sigma}(r)$ is a small term such that for any $\delta>0$,
\[
\mathfrak{S}_{f,\sigma}(r)=O\big(\log T_{f,\sigma}(r)\big)+\delta\,\log r
\,\,\parallel.
\]
\qed
\end{lem}

For an integer $k\geqslant 0$, let $J_k(A)$ denote the $k$-jet space of $A=(\mathbb{C}^*)^n$, which reads as
\[
J_k(A)=A\times J_{k,A}=A\times \mathbb{C}^{nk}.
\]
 There is a natural $A$-action on $J_k(A)$ given   by $a\colon (x,v)\rightarrow (x+a,v)$ for all $x\in A,v\in\mathbb{C}^{nk}$, where ``$+$'' is understood as multiplication. Denote by $J_kf$ the $k$--jet lifting of $f$ and by $X_k(f)$ the Zariski closure of $J_kf$ in the $k$-jet space $J_k(A)$. Let $B:=\St_A\big(X_k(f)\big)$ be the stabilizer group with respect to the natural $A$-action and let $q\colon A\rightarrow A/B$ be the quotient map. Then the jet projection method \cite[Theorem~6.2.6]{Noguchi-Winkelmann2014} together with Lemma~\ref{parabolic lemma on logarithmic form} yield $T_{q\circ f,\sigma}(r)=o\big(T_{f,\sigma}(r)\big)$. Moreover, we can assume $\dim B >0$, otherwise we would get $T_{f,\sigma}(r)=o\big(T_{f,\sigma}(r)\big)$, which is impossible. 

We will first establish a Second Main Theorem for jet liftings. Let $Z$ be an algebraic reduced subvariety of $X_k(f)$. Let $B^0=\St^0_A\big(X_k(f)\big)$ denote the identity component of $B$. Then 
\begin{equation}
\label{key difference}
\dim B^0>0 \quad\text{and}\quad T_{q^{B^0}_k\circ J_kf,\sigma}(r)=o\big(T_{f,\sigma}(r)\big)
\,\,\parallel,
\end{equation}
where $q^{B^0}_k\colon J_k(A)\rightarrow J_k(A)/B^0\cong (A/B^0)\times J_{k,A}$ is the quotient map. This corresponds to \cite[Equation~(6.5.9)]{Noguchi-Winkelmann2014} and hence, we can translate \cite[Theorem~6.5.6]{Noguchi-Winkelmann2014} to the parabolic setting as follows.

\begin{lem}
	\label{smt for jet lifts}
	There exists a compactification $\bar{X}_k(f)$ of $X_k(f)$,  and a positive integer $\ell_0$  such that
	\begin{align}
	m_{J_kf,\sigma}(r,\bar{Z})
	&=o\big(T_{f,\sigma}(r)\big)
	\,\,\parallel,\notag\\
	T_{J_kf,\sigma}(r,\omega_{\bar{Z}})
	&\leqslant
	N^{[\ell_0]}_{J_kf,\sigma}(r,Z)
	+
	o\big(T_{f,\sigma}(r)\big)
	\,\,\parallel,\notag
	\end{align}
	where $\bar{Z}$ denotes the closure of $Z$ in $\bar{X}_k(f)$.	\qed
\end{lem}

Our next goal is to show that the ``impact'' of $J_kf$ on a subvariety of $X_k(f)$ with codimension $\geqslant 2$ is relatively small.

\begin{lem}
	\label{smt for higher comdim}
	Let $Z\subset X_k(f)$ be a subvariety with $\codim_{X_k(f)}Z\geqslant 2$. Then 
	\begin{equation}
	T_{J_kf,\sigma}(r,\omega_{\bar{Z}})
	=o\big(T_{f,\sigma}(r)\big)
	\,\,\parallel.
	\end{equation}
	In particular one has
	\begin{equation}
	N_{J_kf,\sigma}(r,Z) =o\big(T_{f,\sigma}(r)\big)
	\,\,\parallel.
	\end{equation}
\end{lem}
\begin{proof}
This result is an analog of \cite[Theorem~6.5.17]{Noguchi-Winkelmann2014}. Our proof follows the same lines, except a necessary modification in the first reduction. We reduce to the case that $A$ admits a splitting $A=B\times C$ for $B,C$ being semi-abelian varieties of positive dimensions with
\begin{align*}
B\subset\St^0_A\big(X_{k}(f)\big)
\quad
{\scriptstyle ( k\,\geqslant \,0)},
\quad
 T_{q^B\circ f,\sigma}(r)=o\big(T_{f,\sigma}(r)\big)
\,\,\parallel,
\end{align*}
where $q^B\colon A\rightarrow A/B=C$ denotes the projection to the second factor. To do this, we consider the
set of all semi-abelian subvarieties $B\subset A$ such that
\[
T_{q^B\circ f,\sigma}(r)=o\big(T_{f,\sigma}(r)\big)
\,\,\parallel.
\]
We then use \eqref{key difference} and repeat the argument in the proof of \cite[Theorem~6.5.17]{Noguchi-Winkelmann2014}. Note that since we only work with $A=(\mathbb{C}^*)^n$ instead of universal coverings of semi-abelian varieties, the result in \cite[Lemma~6.5.25]{Noguchi-Winkelmann2014} automatically holds. By Lemma~\ref{smt for jet lifts}, it suffices to show that
\[
N^{[1]}_{J_kf,\sigma}(r,Z)=o\big(T_{f,\sigma}(r)\big)
\,\,\parallel.
\]

By induction on the dimension of $Z$, it suffices to check the above estimate for the nonsingular part $Z^{\ns}$ of $Z$. Following the same lines as in \cite[6.5.3]{Noguchi-Winkelmann2014}, we can find a sequence $n(\ell)$ such that $\lim_{\ell\rightarrow\infty}\frac{n(\ell)}{\ell}=0$ and
\[
(\ell+1)\, N^{[1]}_{J_kf,\sigma}(r,Z^{ns})
\leqslant 
n(\ell)\,O\big(T_{f,\sigma}(r)\big)+o\big(T_{f,\sigma}(r)\big)
\,\,\parallel,
\]
which yields the required estimate. This finishes the proof of the Lemma~\ref{smt for higher comdim}.
\end{proof}

\medskip\noindent
\begin{proof}[Proof of Theorem ~\ref{smt-semiabelian, truncation level1}] We follow the argument in~\cite[Section~6.5.4]{Noguchi-Winkelmann2014}. It suffices to consider the case where $Z$ is a reduced Weil divisor on $X_k(f)$ with the  irreducible decomposition $Z=\sum_i Z_i$. Using Lemma~\ref{smt for jet lifts}, we have
\begin{align}
\label{basic start estimate to get truncation level 1}
\quad\quad T_{J_kf,\sigma}(r,\omega_{\bar{Z}})
&\leqslant
N^{[\ell_0]}_{J_kf,\sigma}(r,Z)
+
o\big(T_{f,\sigma}(r)\big)
\,\,\parallel,\notag\\
&\leqslant
N^{[1]}_{J_kf,\sigma}(r,Z)
+
\ell_0\sum_{i<j}N^{[1]}_{J_kf,\sigma}(r,Z_i\cap Z_j)
+
\ell_0\sum_i N^{[1]}_{J_{k+1}f,\sigma}\big(r,J_1(Z_i)\big)
+
o\big(T_{f,\sigma}(r)\big)
\,\,\parallel.
\end{align}

Since $\codim_{X_k(f)}(Z_i\cap Z_j)\geqslant 2$ for $i\not=j$, the second term in the right hand side of \eqref{basic start estimate to get truncation level 1} can be estimated by Lemma~\ref{smt for higher comdim} as
\[
\ell_0\sum_{i<j}N^{[1]}_{J_kf,\sigma}(r,Z_i\cap Z_j)
=o\big(T_{f,\sigma}(r)\big).
\]

We now treat the third term of \eqref{basic start estimate to get truncation level 1}. We consider two cases depending on the position of $B_{k+1}^0:=\St_A^0\big(X_{k+1}(f)\big)$with respect to $\St^0_A(Z_i)$.

\medskip
Case~(1):  $B_{k+1}^0\not\subset\St^0_A(Z_i)$. We have (\cite[Lem. 6.6.50]{Noguchi-Winkelmann2014}):
	\[
	\codim_{X_{k+1}(f)}\big(X_{k+1}(f)\cap J_1(Z_i)\big)\geqslant 2,
	\]
	where we can apply Lemma~\ref{smt for higher comdim} to obtain
	\[
	N^{[1]}_{J_{k+1}f,\sigma}\big(r,J_1(Z_i)\big)
	=
	o\big(T_{f,\sigma}(r)\big).
	\]

Case~(2): $B_{k+1}^0\subset\St^0_A(Z_i)$ We consider the quotient map $q_{k}^{B^0_{k+1}}\colon X_k(f)\rightarrow X_k(f)/B^0_{k+1}$. The image of $Z_i$ under this map is contained in a divisor on $X_k(f)/B^0_{k+1}$, and hence, we can argue as in \cite[Thm. 6.5.6, case (a)]{Noguchi-Winkelmann2014} to get
	\[
	N^{[1]}_{J_{k+1}f,\sigma}\big(r,J_1(Z_i)\big)
	\leqslant
	N_{J_{k+1}f,\sigma}\big(r,J_1(Z_i)\big)
	=o\big(T_{f,\sigma}(r)\big).
	\]

This finishes the proof of Theorem~\ref{smt-semiabelian, truncation level1}.
\end{proof}

\medskip

A family $\{D_i\}_{i=1}^q$ of $q\geqslant n+2$ hypersurfaces in $\mathbb{P}^n(\mathbb{C})$ is said to be {\sl in general position} if any $n+1$ hypersurfaces in this family have empty intersection, namely
\[
\medcap_{i\in I}D_i=\varnothing
\qquad
{\scriptstyle ( I\,\subset\,\{1,\,2,\,\dots,\,q\},\, |I|\,=\,n+1)}.
\]
In \cite{Minru2004}, the author confirms a conjecture of Shiffman by extending the classical Cartan's Second Main Theorem to the case of nonlinear targets. In the parabolic context, the result reads as follows.

\begin{thm}
	\label{Min ru smt-non linear targets in parabolic setting}
	Let $\mathcal{Y}$ be a parabolic Riemann surface with a smooth exhaustion function $\sigma$. Let $\{D_i\}_{i=1}^q$ be a family of $q\geqslant n+2$ hypersurfaces in general position in $\mathbb{P}^n(\mathbb{C})$. Then for any algebraically nondegenerate holomorphic curve $f\colon\mathcal{Y}\rightarrow\mathbb{P}^n(\mathbb{C})$, there exists a positive constant $C$ such that 
	\[
	(q-n-1)\,T_{f,\sigma}(r)\leqslant
	\sum_{i=1}^q \dfrac{N_{f,\sigma}(r,D_i)}{\deg(D_i)}
	+
	C\,\mathfrak{X}_\sigma(r)
	+
	o\big(T_{f,\sigma}(r)\big)
	\,\,\parallel.
	\]
	\qed
\end{thm}

The proof follows the same lines as in \cite{Minru2004}, where the filtration method of Corjava-Zannier \cite{Corvaja-Zannier2004} was employed to reduce the problem to the linear case \cite{Minru1997, Vojta97}.

\section{\bf A smooth exhaustion function on some  parabolic Riemann surface}\label{sect-smoothing}

In this section, we construct a piecewise smooth exhaustion function $\hat{\sigma}$ on the parabolic Riemann surface $\mathbb{C}\backslash\mathcal{E}$, where $\mathcal{E}=\{a_j\}_{j=1}^{\infty}$ is a discrete countable set of points on $\mathbb{C}$. Then we describe a smooth exhaustion function $\sigma$ close to $\hat{\sigma}$. Details are presented in the Appendix.

We arrange $a_j$ so that
\[
|a_1|\leqslant |a_2|\leqslant\dots.
\]
Take $r_j\in(0,1)$ sufficiently small such that
\begin{itemize}
\smallskip
\item[$\bullet$]
the discs $\mathbb{D}(a_j,2\,r_j)$ are disjoint;
\smallskip
\item[$\bullet$] the sum $\sum\limits_{j\geqslant 1}r_j<+\infty$.
\end{itemize}

Let $\mathcal{Y}:=\mathbb{C}\backslash\{a_j\}_{j=1}^{\infty}$ and define a piecewise smooth exhaustion function $\hat{\sigma}:\mathcal{Y}\rightarrow[1,+\infty)$ by
\begin{align}\label{def-exhaustion}
\hat{\sigma}(z):=\exp
\Big(\log^+|z|+\sum\limits_{j=1}^{\infty}r_j\,\log^+\tfrac{r_j}{|z-a_j|}\Big).
\end{align}
In other words
\[
\hat{\tau}:=\log\hat{\sigma}=\log^+|z|+\sum\limits_{j=1}^{\infty}r_j\,\log^+\tfrac{r_j}{|z-a_j|}.
\]
Obviously, the function $\hat{\tau}$
\begin{itemize}
\smallskip
\item[$\bullet$]
takes value in $[0,+\infty)$;
\smallskip
\item[$\bullet$]
is continuous on $\mathcal{Y}$;
\smallskip
\item[$\bullet$]
is smooth, indeed harmonic, outside the circle $S(0,1):=\big\{|z|=1\big\}$ and the disjoint circles $S(a_j,r_j):=\big\{|z-a_j|=r_j\big\}$.
\end{itemize}

By the Poincar\'e-Lelong formula (cf. e.g. \cite[Theorem~2.2.16]{Noguchi-Winkelmann2014}), it is clear that
\begin{equation}\label{eqn:ddctau}
	\dif\dif^c\hat{\tau}=\frac{1}{2}\,\nu(0,1)+\frac{r_j}{2}\sum\limits_{j=1}^{\infty}\big(\nu(a_j,r_j)-\delta_{a_j}\big)
\end{equation}
is a distribution of order $0$ and locally of finite mass. Here $\nu(a_j,r_j)$ is the Haar measure on the circle $S(a_j,r_j)$.

Use the notations
\[
B^{\hat{\sigma}}_r:= \big\{z\in \mathcal Y : \hat{\sigma}(z)< r\big\},\qquad\qquad
S^{\hat{\sigma}}_r:= \big\{z\in \mathcal Y : \hat{\sigma}(z)= r\big\},
\]
for the $\hat{\sigma}$-ball of radius $r$ and its boundary. For $r>0$, the boundary $S^{\hat{\sigma}}_r$ is a piece-wise smooth curve. It has
\[
\#\big\{j~:~|a_j|+r_j<r\big\}+1
\]
many connected components. Non-smooth points come from two cases:
\begin{enumerate}
\smallskip
\item
when $r=|a_j|+r_j$ for some $j$, there is one non-smooth point which is the tangent point of $S(0,r)$ to $S(a_j,r_j)$;
\smallskip
\item
when $r\in\big(|a_j|-r_j,|a_j|+r_j\big)$ for some $j$, there are two non-smooth points which are the intersection points of $S(0,r)$ and $S(a_j,r_j)$.
\end{enumerate}

\begin{figure}[!htbp]
	\begin{center}
		\includegraphics[width=0.3\linewidth]{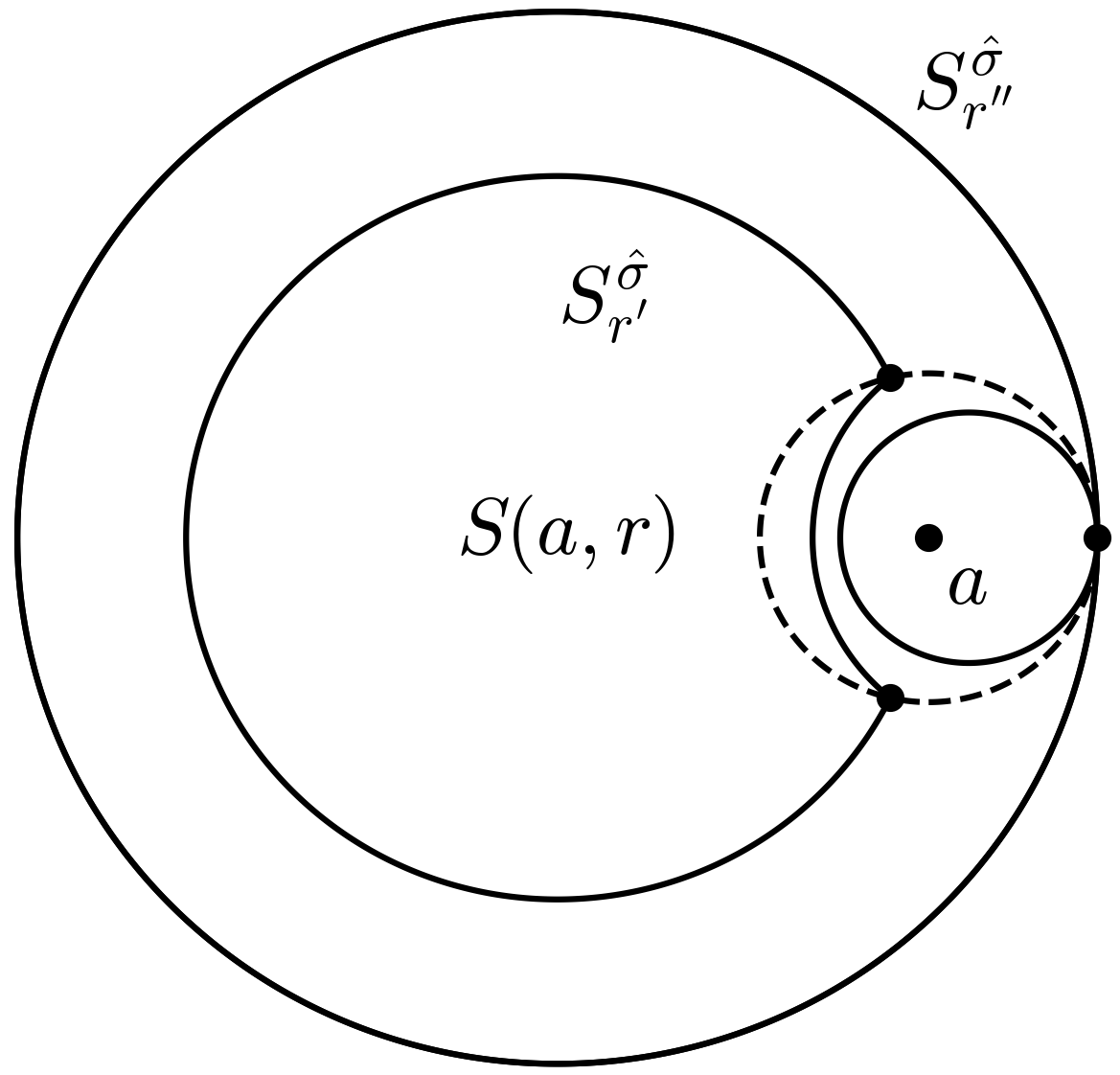} 
		\caption{The non-smooth points on $S^{\hat{\sigma}}_{r'}$ of case ($1$) and on $S^{\hat{\sigma}}_{r'}$ of case ($2$).}
	\end{center}
\end{figure}

Now we describe an exhaustion function $\sigma$ of $\mathcal{Y}$. An explicit construction and the proof of the Lemma property will be given in the Appendix.

\begin{lem}\label{lem-smoothing-1}
There is a smooth exhaustion function $\sigma\geqslant \hat{\sigma}$ such that the difference $\sigma-\hat{\sigma}$ is supported on
\[
\mathrm{Supp}(\sigma-\hat{\sigma})
\subset
U:=
\big(A(0,\tfrac{1}{2},\tfrac{3}{2})\backslash\mathcal{E}\big)\cup
\bigcup_{j=1}^{\infty}A(a_j,\tfrac{1}{2}r_j,\tfrac{3}{2}r_j),
\]
where $A(a_j,\tfrac{1}{2}r_j,\tfrac{3}{2}r_j):=\{z\in\mathcal{Y} : \tfrac{1}{2}r_j\leqslant |z-a_j|\leqslant \tfrac{3}{2}r_j\}$ are pairwise disjoint annuli. Moreover, for $z\notin\bigcup_{j=1}^{\infty}\overline{D(a_j,\frac{3}{2}r_j})$ with $\sigma(z)\geqslant \frac{3}{2}$, one has $\sigma(z)=\hat\sigma(z)$.
\end{lem}

\begin{figure}[!htbp]
	\begin{center}
		\includegraphics[width=0.8\linewidth]{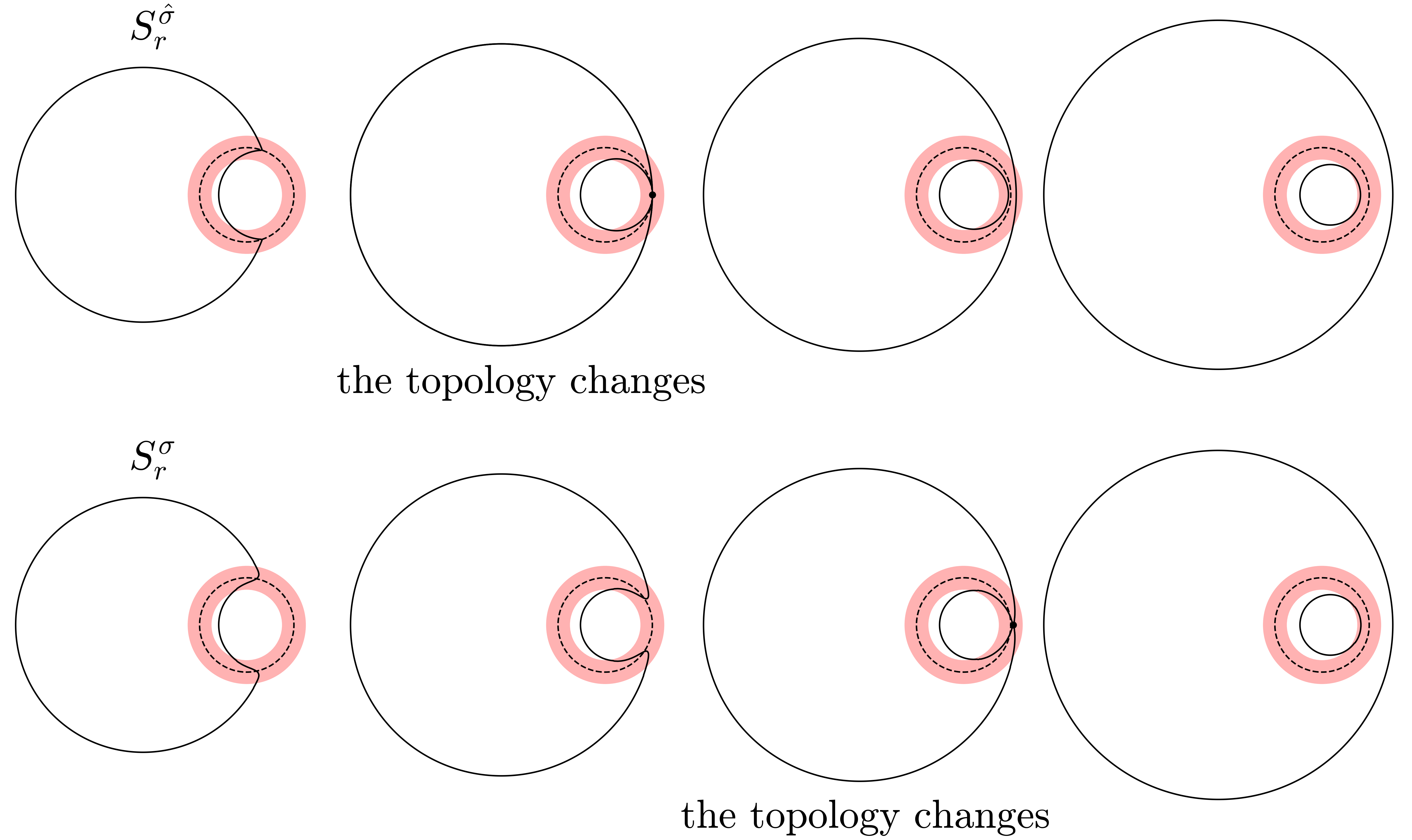} 
		\caption{The curves $S^{\hat{\sigma}}_{r}$ and $S^{\sigma}_r$ as $r$ increases.}
	\end{center}
\end{figure}

Let $B^{\sigma}_r:=\{z\in\mathcal{Y} : \sigma(z)<r\}$ be the $\sigma$-ball of radius r. Then the Lemma above implies
\[
B^\sigma_t\subset \mathbb{D}_t
 \qquad
{\scriptstyle (t\,\geqslant\, 1)}
\]
and
\[
\mathbb{D}_t\backslash\Big(\bigcup_{j=1}^{\infty} \mathbb{D}(a_j,\tfrac{3}{2}r_j)\Big)\subset B^\sigma_t
\qquad
{\scriptstyle (t\,\geqslant\, \tfrac{3}{2})}.
\]

By the argument in~\cite[Proposition~3.3, pp.~32--33]{Sibony-Paun2021}, the weighted Euler characteristic saitsfies
\[
\mathfrak{X}^+_\sigma(r)=\mathfrak{X}_\sigma(r)+O(\log r)=\int_{t=1}^r\#\{j \colon |a_j|<t\}\,\frac{\dif t}{t}+O(\log r).
\]

The following Lemma ensures that the Parabolic Logarithmic Derivative Lemma~\ref{ldl} holds for $\sigma$ (see~\cite[Remark~3.9]{Sibony-Paun2021}).
\begin{lem}\label{lem-smoothing-2}The smooth $2$-form $\dif\dif^c\log\sigma$ defines an order $0$ distribution of finite mass on $\mathcal{Y}$.
\qed
\end{lem}

\section{\bf Proof of the Main Theorem}
\label{Proof of the main theorem}

Let $f\colon\mathbb{C}\rightarrow\mathbb{P}^n(\mathbb{C})$ be a holomorphic curve and let $D=\sum_{i=1}^{n+1}D_i$ be a simple normal crossing divisor on $\mathbb{P}^n(\mathbb{C})$. Let $Q_i$ be the defining homogeneous polynomial of $D_i$ with degree $d_i$. Let $F:\mathbb{P}^n(\mathbb{C})\rightarrow\mathbb{P}^n(\mathbb{C})$ be the endomorphism of degree $d=\lcm(d_1,\dots,d_{n+1})$ defined by
\begin{eqnarray}
\label{endomorphism F definition}
F(z):=\big[Q^{m_1}_1(z)\,\colon\,\dots\,\colon\,Q^{m_{n+1}}_{n+1}(z)\big],
\end{eqnarray}
where $m_i=\tfrac{d}{d_i}$ for $1\leqslant i\leqslant n+1$.
By construction, $F$ maps  $\mathbb{P}^n(\mathbb{C})\setminus D$ to $(\mathbb{C}^*)^n$. The critical points of $F$ consists of hypersurfaces $D_i$ (if $m_i\geqslant 2$) and a hypersurface $\mathcal{V}$ of degree $\sum_{i=1}^{n+1}d_i-(n+1)>0$ defined by 
\[
M(z):=\det
\dfrac{\partial (Q_1, \dots, Q_{n+1})}{\partial (z_0, \dots, z_{n})}=0.
\]

\begin{lem}\label{lem:transversal}
The hypersurface $\mathcal V$ is in general position with $\{D_i\}_{i=1}^{n+1}$, if and only if the hypersurfaces $\{D_i\}_{i=1}^{n+1}$ intersect transversally.
\end{lem}
\proof The {\em if} part was proven in~\cite[Section~5]{Guo-Sun-Wang-2022}. For the {\em only if} part, without loss of generality we may assume that there exists some $p\in D_1\cap\dots\cap D_n$ such that
\begin{itemize}
\item either $p$ is a non-smooth point in some $D_k$, $1\leqslant k\leqslant n$, i.e.
\[
\left(\frac{\partial Q_k}{\partial z_0}(p),\dots,\frac{\partial Q_k}{\partial z_n}(p)\right)=(0,\dots,0);
\]
\item or the normal vectors $\{(\frac{\partial Q_k}{\partial z_0}(p),\dots,\frac{\partial Q_k}{\partial z_n}(p))\}_{1\leqslant k\leqslant n}$ of the tangent spaces are linearly dependent.
\end{itemize}
In both cases,
\[
\rank\left(\frac{\partial Q_i}{\partial z_j}(p)\right)_{1\leqslant i\leqslant n,0\leqslant j\leqslant n}<n,
\]
i.e. each $n$-minor of the $n\times (n+1)$ matrix has determinant $0$.

The point $p$ has a homogeneous representation $[p_0:\dots:p_n]$. There is some $s\in\{0,\dots,n\}$ such that $p_s\neq 0$. Following Guo-Sun-Wang's argument~\cite{Guo-Sun-Wang-2022}, by using the Euler formula
\[
\sum\limits_{j=1}^n\frac{\partial Q_i}{\partial z_j}z_j=d_i\cdot Q_i,
\]
we have
\[
p_s\,M(p)=
\det
\begin{pmatrix}
\frac{\partial Q_1}{\partial z_0}(p) & \dots & d_1\cdot Q_1(p) & \dots & \frac{\partial Q_1}{\partial z_n}(p)\\
\vdots & \ddots & \vdots & \ddots & \vdots\\
\frac{\partial Q_{n+1}}{\partial z_0}(p) & \dots & d_{n+1}\cdot Q_{n+1}(p) & \dots & \frac{\partial Q_{n+1}}{\partial z_n}(p)\\
\end{pmatrix}.
\]
Noting that $Q_1(p)=\dots=Q_n(p)=0$ since $p\in D_1\cap\dots\cap D_n$. Hence
\[
p_s\,M(p)=(-1)^{n+s}\,d_{n+1}\cdot Q_{n+1}(p)\,\det\left(\frac{\partial Q_i}{\partial z_j}(p)\right)_{1\leqslant i\leqslant n,0\leqslant j\leqslant n, j\neq s}=0.
\]
Thus $M(p)=0$, i.e. $\mathcal{V}$ and $\{D_i\}_{i=1}^n$ intersect at $p$. We conclude that they are not in general position.\qed

\begin{figure}[!htbp]
	\begin{center}
		\includegraphics[width=0.4\linewidth]{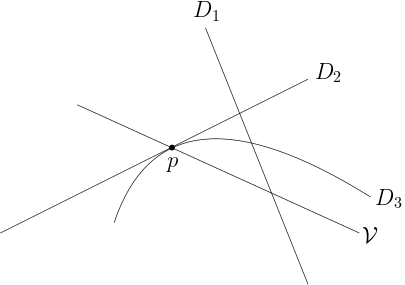} 
		\caption{When $D_2$ and $D_3$ intersect non-transversally, $\mathcal{V}$ and $D_2$, $D_3$ are not in general position}
	\end{center}
\end{figure}

Set $g:=F\circ f$. The image $F(\mathcal{V})$ is an algebraic variety, denoted by $\mathcal{W}$. 
It is clear that
\[
T_g(r)=O\big(T_f(r)\big).
\]

As an illustrated example, one first looks at the case of projective plane and $D$ is the union of three conics $D_1, D_2, D_3$ in $\mathbb{P}^2(\mathbb{C})$ in general position. Then for {\rm every} $z \in f^{-1}(\mathcal{V})$ one has
	\[
	\ord_zg^*\mathcal{W}\geqslant \ord_zf^*\mathcal{V}+1.
	\]
	\begin{figure}[!htbp]
		\begin{center}
			\scalebox{0.8}{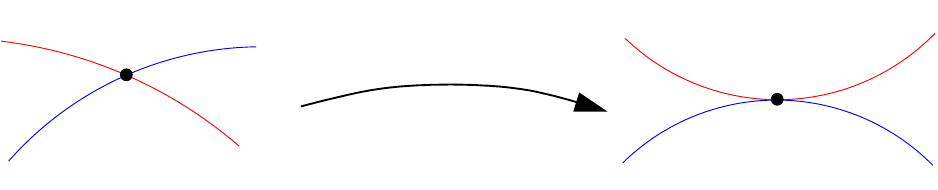} 
		\end{center}
	\end{figure}

Indeed, since $g = F \circ f$, one always has $\ord_zg^*\mathcal{W}\geqslant \ord_zf^*\mathcal{V}$. Thus one only needs to exclude the possibility that $\ord_zg^*\mathcal{W}= \ord_zf^*\mathcal{V}$. In the simple case when $\ord_zg^*\mathcal{W}= \ord_zf^*\mathcal{V}=1$, this means exactly that $F$ has maximal rank at the point $p= f(z)$, which contradicts the definition of $\mathcal{V}$ as the critical set of $F$.

%For general $D$ in $\mathbb{P}^n(\mathbb{C})$ and for general $\ord_zf^*\mathcal{V}$, one obtains the following key property of the hypersurface $\mathcal{V}$.

\begin{pro}
	\label{compare f and g counting}
 There exists a proper subvariety $\mathcal{Z}$ of $\mathcal{V}$ such that for every $z \in f^{-1}(\mathcal{V} \setminus \mathcal{Z})$, one has
\[
\ord_zf^*\mathcal{V}
\leqslant
\ord_zg^*\mathcal{W}-1.
	%\eqno{{\scriptstyle (\forall\,z\in\mathbb{C})}}.
	\]
\end{pro}

\begin{proof}
%  One can derive this inequality by using a general theorem of Noguchi-Winkelmann-Yamanoi (cf. \cite[Theorem~3.12, (iv)]{NWA2007}). Here our setting is simpler and we provide an elementary proof for readers' convenience.
By our construction, the hypersurface $\mathcal{V}$ is contained in the support of the ramification divisor of the endomorphism $F:\, \mathbb{P}^n(\mathbb{C}) \to \mathbb{P}^n(\mathbb{C})$. Putting
\begin{align*}
  \mathcal{Z}_1&= \mathcal{V} \cap \mathrm{Supp}\,D,\\
  \mathcal{Z}_2&=\Sing(\mathcal{V})\,\cup\,F^{-1}(\Sing(\mathcal{W})).
\end{align*}
Let $\mathcal{Z}=\mathcal{Z}_1\cup\mathcal{Z}_2$. Then for any point $p \in \mathcal{V} \setminus \mathcal{Z}$, there exist local coordinate systems $(x_1,\dots, x_n)$ about $p$ and $(y_1,\dots, y_n)$ about $q=F(p)$ such that locally one has
\[
  \mathcal{V} = \{x_1=0\},
   \qquad
 \mathcal{W} = \{y_1=0\},
\]
\[
F(x_1,\dots, x_n) =: (y_1,\dots, y_n)=(x_1^m,x_2,\dots,x_n).
\]
Here, by our construction of $\mathcal{V}$, at point $p$ the associated $m\geqslant 2$. Thus locally we have  $F^*\mathcal{W} = m\,\mathcal{V}$. Consequently,
\[
\mathrm{ord}_z g^*\mathcal{W} = \mathrm{ord}_z f^*(F^*\mathcal{W}) = m\, \mathrm{ord}_z f^*\mathcal{V} = \mathrm{ord}_z f^*\mathcal{V} + (m-1)\, \mathrm{ord}_z f^*\mathcal{V} \geqslant \mathrm{ord}_z f^*\mathcal{V} +1
\]
for every $z \in f^{-1}(\mathcal{V} \setminus \mathcal{Z})$.
\end{proof}

Now put $\mathcal{E}:=f^{-1}(D)$, which is a discrete countable set of points in $\mathbb{C}$. We arrange $\mathcal{E}=\{a_j\}_{j=1}^{\infty}$ so that $|a_1|\leqslant |a_2|\leqslant\dots$. Note that $\#\{\mathbb{D}_t\cap\mathcal{E}\}$ is exactly $n_f^{[1]}(t,D)$, which is finite. Denote by $\widetilde{f}$, $\widetilde{g}$ the restrictions of $f$, $g$ on $\mathcal{Y}:=\mathbb{C}\setminus\mathcal{E}$ respectively. Consider the exhaustion function $\hat{\sigma}$ defined in the previous section. Take a smoothing $\sigma$ of $\hat{\sigma}$ as in the Appendix. Denote by $B^\sigma_r$ and $S^\sigma_r$ the $\sigma$-ball of radius $r$ and its boundary. The construction in the appendix ensures $\sigma(z)\geqslant\hat{\sigma}(z)\geqslant |z|$ on $\mathcal{Y}$, hence $B^\sigma_r\subset\mathbb{D}_r$ and
\begin{equation}
\label{mean euler characterisitic vs counting}
\mathfrak{X}_{\sigma}(r)
\leqslant N^{[1]}_f(r,D)+O(\log r)
\qquad
{\scriptstyle (r\,>\,1)}.
\end{equation}

Suppose on the contrary that \eqref{defect relation} does not hold. Then the weighed Euler characteristic $\mathfrak{X}_{\sigma}(r)$ satisfies
\begin{equation}
\label{euler characteristic and order of f}
\limsup_{r\rightarrow \infty}\dfrac{\mathfrak{X}_{\sigma}(r)}{T_f(r)}=0.
\end{equation}

Since
\[
B^\sigma_t\subset \mathbb{D}_t
 \qquad
{\scriptstyle (t\,\geqslant\, 1)}
\]
and
\[
\mathbb{D}_t\backslash\Big(\bigcup_{j=1}^{\infty} \mathbb{D}(a_j,\tfrac{3}{2}r_j)\Big)\subset B^\sigma_t
\qquad
{\scriptstyle (t\,\geqslant\, \tfrac{3}{2})},
\]
one has
\[
\aligned
0\leqslant T_f(r)-T_{\widetilde{f},\sigma}(r)
&\leqslant\int_{\tfrac{3}{2}}^r\frac{\dif t}{t}\int_{\Big(\bigcup_{j=1}^{\infty} \mathbb{D}(a_j,\tfrac{3}{2}r_j)\Big)\cap \mathbb{D}_t}f^*\omega+O(1)\\
&\leqslant\int_1^r\frac{\dif t}{t}\int_{\bigcup_{j=1}^{\infty} \mathbb{D}(a_j,\tfrac{3}{2}r_j)}f^*\omega+O(1)\\
&=\int_1^r\frac{\dif t}{t}\Big(\sum\limits_{j=1}^{\infty}\int_{\mathbb{D}(a_j,\tfrac{3}{2}r_j)}f^*\omega\Big)+O(1),
  \qquad
{\scriptstyle (r\,\geqslant\, 1)}.
\endaligned
\]
Recall (\ref{def-exhaustion}) that the radius $r_j>0$ can be chosen arbitrarily small. For our purpose, for each $j\geqslant 1$, we choose $r_j>0$ sufficiently small so that $\int_{\mathbb{D}(a_j,\tfrac{3}{2}r_j)}f^*\omega<2^{-j}$. Hence the above estimate yields
\[
0\leqslant T_f(r)-T_{\widetilde{f},\sigma}(r)\leqslant \log r+O(1),
  \qquad
{\scriptstyle (r\,\geqslant\, 1)}.
\]
This together with \eqref{euler characteristic and order of f} implies
\[
\limsup_{r\rightarrow \infty}\dfrac{\mathfrak{X}_\sigma(r)}{T_{\widetilde{f},\sigma}(r)}=0.
\]
Hence the technical assumption \eqref{assumption of weight euler charactersistic} is satisfied, which allows us to use all of the obtained results in the parabolic setting. First, applying Theorem~\ref{Min ru smt-non linear targets in parabolic setting}, we receive
\begin{equation}
\label{min ru smt application}
T_{\widetilde{f},\sigma}(r)
\leqslant
\dfrac{N_{\widetilde{f},\sigma}(r,\mathcal{V})}{\deg\mathcal{V}}
+o\big(T_{\widetilde{f},\sigma}(r)\big)
\,\,\parallel.
\end{equation}

Next, using Corollary~\ref{tangency not often} for  $\widetilde{g}$, we get
\begin{equation}
\label{counting tangent point for g}
N_{\widetilde{g},\sigma}(r,\mathcal{W})
-
N^{[1]}_{\widetilde{g},\sigma}(r,\mathcal{W})
=
o\big(T_{\widetilde{g},\sigma}(r)\big)
\,\,\parallel.
\end{equation}

On the other hand, it follows from Theorem~\ref{smt-semiabelian, truncation level1} and Proposition~\ref{compare f and g counting} that
\begin{equation}
\label{counting function tilde f and tilde g}
N_{\widetilde{f},\sigma}(r,\mathcal{V})
\leqslant
N_{\widetilde{g},\sigma}(r,\mathcal{W})
-
N^{[1]}_{\widetilde{g},\sigma}(r,\mathcal{W})
+o\big(T_{\widetilde{f},\sigma}(r)\big)
\,\,\parallel.
\end{equation}

Combining \eqref{min ru smt application}, \eqref{counting tangent point for g}, \eqref{counting function tilde f and tilde g}, one has

\begin{align*}
T_{\widetilde{f},\sigma}(r)
&\leqslant
\dfrac{N_{\widetilde{f},\sigma}(r,\mathcal{V})}{\deg\mathcal{V}}
+o\big(T_{\widetilde{f},\sigma}(r)\big)
\,\,\parallel\\
&\leqslant
\dfrac{N_{\widetilde{g},\sigma}(r,\mathcal{W})
	-
	N^{[1]}_{\widetilde{g},\sigma}(r,\mathcal{W})}{\deg\mathcal{V}}
+o\big(T_{\widetilde{f},\sigma}(r)\big)
\,\,\parallel\\
&=
o\big(T_{\widetilde{g},\sigma}(r)\big)
+o\big(T_{\widetilde{f},\sigma}(r)\big)
\,\,\parallel,
\end{align*}
which is a contradiction. This finishes the proof of the Main Theorem.

\begin{rmk}
	In the case where $f\colon\mathbb{C}\rightarrow\mathbb{P}^2(\mathbb{C})$ is an algebraically nondegenerate holomorphic curve and where $\mathcal{C}$ is the collection of two lines and one conic in $\mathbb{P}^2(\mathbb{C})$,
	in a private note, Noguchi obtained a Second Main Theorem of the form
\[
T_f(r)
\leqslant C\, N_f(r,\mathcal{C})
+
\big[
N_f^{[2]}(r,\mathcal{V})-N_f^{[1]}(r,\mathcal{V})
\big]
+o\big(T_f(r)\big)
\,\,\parallel,
\]
where $\mathcal{V}$ is the critical curve of the endomorphism defined as above, and $C>0$ is some constant. Although the right hand side of the above inequality involves a quantity depending on $\mathcal{V}$ (which actually counts the number of tangent points of $f$ and $\mathcal{V}$), this term is negligible when $f$ omits $\mathcal{C}$.
\end{rmk}

\begin{rmk}
	Our result can be extended to the case of entire holomorphic curves into algebraic varieties of log-general type $X$ with $\overline{q}(X)=\dim X$ by similar argument.
\end{rmk}

\appendix
\setcounter{section}{-1}
\section{\bf A detailed construction of one smooth exhaustion function}\label{append}

We provide an explicit construction of  \cite[pp.~32--33,  Example~(2)]{Sibony-Paun2021}, precisely, a smooth exhaustion function $\sigma$ on the parabolic surface $\mathcal{Y}:=\mathbb{C}\backslash\{a_j\}_{j=1}^{\infty}$ which  satisfies Lemma~\ref{lem-smoothing-1} and Lemma~\ref{lem-smoothing-2}.

\begin{proof}[Proof of Lemma~\ref{lem-smoothing-1}] Define
\[
h(r):=\left\{
\aligned
&0 &
  \qquad
{\scriptstyle (r\,\leqslant\, \tfrac{3}{4})},
\\
&\frac{1}{4\pi}\frac{1}{1+e^{\frac{r-1}{(r-1)^2-1/16}}}  &
  \qquad
{\scriptstyle (\tfrac{3}{4}\,<\,r\,<\,\tfrac{5}{4})},
\\
&\frac{1}{4\pi} &
  \qquad
{\scriptstyle (r\,\geqslant\, \tfrac{5}{4})}.
\endaligned\right.
\]
The function $h(r)$ is bounded and agrees with $\frac{1}{4\pi}\mathds{1}_{r>1}$ outside $[\tfrac{3}{4},\tfrac{5}{4}]$.
\begin{figure}[!htbp]
	\begin{center}
		\includegraphics[width=0.4\linewidth]{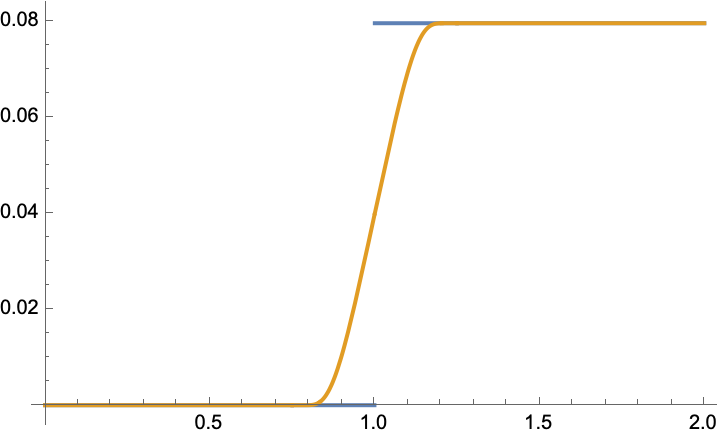} 
		\caption{Graph of $h(r)$ and $\frac{1}{4\pi}\mathds{1}_{r>1}$.}
	\end{center}
\end{figure}

By symmetry, for $r\geqslant \tfrac{5}{4}$, the integration
\[
\int_{0}^r 4\pi\,h(s)\dif s=\int_{0}^r \mathds{1}_{s>1}\dif s=r-1.
\]
However, the integration
\[
\int_{0}^r 4\pi\,h(s)\frac{\dif s}{s}<\int_{0}^r \mathds{1}_{s>1}\frac{\dif s}{s}=\log r,
\]
since $1/s$ is strictly decreasing on $(0,+\infty)$. We need a small translation $c\in(0,\tfrac{1}{2})$ to ensure that the primitive
\[
H(r):=4\pi\int_{0}^r h\big(s-c)\big)\frac{\dif s}{s}
\]
agrees with $\log^+ r=\int_{0}^r \mathds{1}_{r>1}\frac{\dif s}{s}$ outside $[\tfrac{1}{2},\tfrac{3}{2}]$.
\begin{figure}[!htbp]
	\begin{center}
		\includegraphics[width=0.4\linewidth]{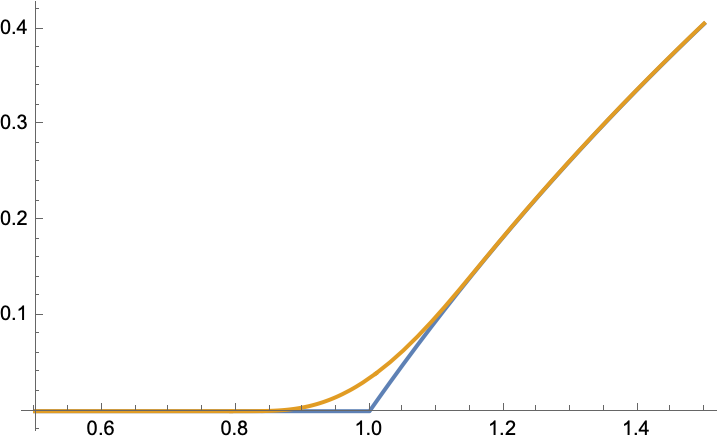} 
		\caption{Graph of $H(r)$ and $\log^+r$.}
	\end{center}
\end{figure}

Thus we get a smoothing $H(r)$ of $\log^+r$ with
\begin{align}
\label{approx-log-by-H}
0\leqslant H(r)- \log^+ r\leqslant \log^+\tfrac{3}{2}<\tfrac{1}{2}.
\end{align}
Together with the monotonicity of $H$, one has $H(r)=\log^+r$ when $H(r)\geqslant \log\frac{3}{2}$.

Define
\[
\tau:=H(|z|)+\sum\limits_{j=1}^\infty r_j\,H\big(|\tfrac{r_j}{z-a_j}|\big),
\qquad
\sigma:=\exp(\tau).
\]
Then $\sigma\geqslant\hat{\sigma}$ and the difference $\sigma-\hat{\sigma}$ is supported on
\[
\mathrm{Supp}(\sigma-\hat{\sigma})
\subset
U:=\big(A(0,\tfrac{1}{2},\tfrac{3}{2})\backslash\mathcal E\big)\cup
\bigcup_{j=1}^{\infty}A(a_j,\tfrac{1}{2}r_j,\tfrac{3}{2}r_j).
\]
where $A(a_j,\tfrac{1}{2}r_j,\tfrac{3}{2}r_j)$ are pairwise disjoint.

For $z\notin\bigcup_{j=1}^{\infty}\overline{D(a_j,\frac{3}{2}r_j})$, $\log^+|\frac{r_j}{z-a_j}|=H(|\frac{r_j}{z-a_j}|)=0$ for each $j$. Thus $\tau(z)=H(|z|)$ and $\hat{\tau}(z)=\log^+|z|$, which are equal when $\tau(z)\geqslant \log\frac{3}{2}$. Taking exponential, $\sigma=\hat{\sigma}$ when $\sigma\geqslant \frac{3}{2}$.

\end{proof}

\begin{proof}[Proof of Lemma~\ref{lem-smoothing-2}] In polar coordinates $z=a+r\,e^{i\theta}$ for some $a\in\mathbb{C}$, for a smooth function $\phi$ one has~\cite[pp.~2]{Noguchi-Winkelmann2014}
\[
\dif^c\phi=\frac{1}{4\pi}\left(r\frac{\partial\phi}{\partial r}\dif\theta-\frac{1}{r}\frac{\partial\phi}{\partial \theta}\dif r\right).
\]
The smooth $2$-form $\dif\dif^c\log\sigma$ is supported on $U$ since $\log\sigma$ is harmonic elsewhere. On the annulus $A(0,\tfrac{1}{2},\tfrac{3}{2})$, in polar coordinates $z=re^{i\theta}$ one has
\[
\dif\dif^c H(|z|)=\dif\left(\frac{r}{4\pi}\frac{\partial H(r)}{\partial r}\dif\theta\right)=\dif\left(h(r-c)\dif\theta\right)=O(1)\dif r\wedge\dif\theta=O(1)\,r\dif r\wedge\dif\theta
\]
since $r\in[\tfrac{1}{2},\tfrac{3}{2}]$ is bounded. Thus $\dif\dif^c H(|z|)$ is of finite mass on $A(0,\tfrac{1}{2},\tfrac{3}{2})\backslash\mathcal{E}$.

On the annulus $A(a_j,\tfrac{1}{2}r_j,\tfrac{3}{2}r_j)$, in polar coordinates $z=a_j+re^{i\theta}$ one has
\[
\aligned
\dif\dif^c
\sum\limits_{j=1}^\infty r_j\,H\big(|\tfrac{r_j}{z-a_j}|\big)
&=\dif\left(\frac{r}{4\pi}\frac{\partial r_j\,H(\tfrac{r_j}{r})}{\partial r}\dif\theta\right)\\
&=\dif\left(-r_j\,h(\tfrac{r_j}{r}-c)\dif\theta\right)=O(r_j)\dif r\wedge\dif\theta=O(1)\,r\dif r\wedge\dif\theta.
\endaligned
\]
Since $\sum\limits_{j=1}^{\infty}r_j^2\leqslant\sum\limits_{j=1}^{\infty} r_j<+\infty$, the support $U$ is of finite Lebesgue measure. We conclude that $\dif\dif^c\log\sigma$ is of finite mass.
\end{proof}

\begin{center}
	\bibliographystyle{alpha}
	\thispagestyle{empty}
	\bibliography{references}
\end{center}

\end{document}

%% file: tangent.pdf_t
\begin{picture}(0,0)%
\includegraphics{tangent.pdf}%
\end{picture}%
\setlength{\unitlength}{3947sp}%
\begingroup\makeatletter\ifx\SetFigFont\undefined%
\gdef\SetFigFont#1#2#3#4#5{%
  \reset@font\fontsize{#1}{#2pt}%
  \fontfamily{#3}\fontseries{#4}\fontshape{#5}%
  \selectfont}%
\fi\endgroup%
\begin{picture}(7488,1506)(3495,-3664)
\put(8626,-2386){\makebox(0,0)[lb]{\smash{{\SetFigFont{12}{14.4}{\familydefault}{\mddefault}{\updefault}{\color[rgb]{1,0,0}$g(\mathbb{C})$}%
}}}}
\put(6976,-2686){\makebox(0,0)[lb]{\smash{{\SetFigFont{12}{14.4}{\familydefault}{\mddefault}{\updefault}{\color[rgb]{0,0,0}$F$}%
}}}}
\put(3676,-3586){\makebox(0,0)[lb]{\smash{{\SetFigFont{12}{14.4}{\familydefault}{\mddefault}{\updefault}{\color[rgb]{0,0,1}$\mathcal{V}$}%
}}}}
\put(8641,-3541){\makebox(0,0)[lb]{\smash{{\SetFigFont{12}{14.4}{\familydefault}{\mddefault}{\updefault}{\color[rgb]{0,0,1}$\mathcal{W}$}%
}}}}
\put(3631,-2341){\makebox(0,0)[lb]{\smash{{\SetFigFont{12}{14.4}{\familydefault}{\mddefault}{\updefault}{\color[rgb]{1,0,0}$f(\mathbb{C})$}%
}}}}
\end{picture}%